\documentclass[a4paper,11pt,sort&compress]{elsarticle}
\usepackage{amscd}
\usepackage{epstopdf}
\usepackage{color}
\usepackage{latexsym}
\usepackage{epsf}
\usepackage{multicol}
\usepackage{ifpdf}
\usepackage{amsmath,amsfonts,amssymb,epsfig,subfigure}
\usepackage{mathrsfs}
\usepackage{stmaryrd}
\usepackage{graphicx,amsmath}
\usepackage{latexsym}
\usepackage{verbatim}
\usepackage{listings}

\ifpdf
\usepackage[%
  pdftitle={Instructions for use of the document class
    elsart},%
  pdfauthor={},%
  pdfsubject={The preprint document class elsart},%
  pdfkeywords={instructions for use, elsart, document class},%
  pdfstartview=FitH,%
  bookmarks=true,%
  bookmarksopen=true,%
  breaklinks=true,%
  colorlinks=true,%
  linkcolor=blue,anchorcolor=blue,%
  citecolor=blue,filecolor=blue,%
  menucolor=blue,pagecolor=blue,%
  urlcolor=blue]{hyperref}
\else
\usepackage[%
  breaklinks=true,%
  colorlinks=true,%
  linkcolor=blue,anchorcolor=blue,%
  citecolor=blue,filecolor=blue,%
  menucolor=blue,pagecolor=blue,%
  urlcolor=blue]{hyperref}
\fi
\makeatletter
\def\elsartstyle{%
    \def\normalsize{\@setfontsize\normalsize\@xiipt{14.5}}
    \def\small{\@setfontsize\small\@xipt{13.6}}
    \let\footnotesize=\small
    \def\large{\@setfontsize\large\@xivpt{18}}
    \def\Large{\@setfontsize\Large\@xviipt{22}}
    \skip\@mpfootins = 18\p@ \@plus 2\p@
    \normalsize
} \@ifundefined{square}{}{\let\Box\square}
\makeatother

\newtheorem{theorem}{Theorem}[section]

\newtheorem{lemma}[theorem]{Lemma}
\newtheorem{rem}[theorem]{Remark}
\newtheorem{expl}[theorem]{Example}

\makeatletter \@addtoreset{equation}{section}

\allowdisplaybreaks
\newcommand{\E}{\mathbb{E}}
\newcommand{\PP}{\mathbb{P}}
\newcommand{\RR}{\mathbb{R}}

\def\nn{\nonumber}

\def\lf{\left} \def\rt{\right}\def\t{\triangle} 
\def\la{\label}\def\be{\begin{equation}}  \def\ee{\end{equation}}

\usepackage[text={6.5in,9in},centering]{geometry}

\journal{Journal of \LaTeX\ Templates}

\begin{document}

\begin{frontmatter}

\title{First order strong convergence and extinction of positivity preserving logarithmic  truncated Euler-Maruyama method  for the stochastic SIS epidemic model\tnoteref{mytitlenote}}
\tnotetext[mytitlenote]{This work was supported by National Natural Science Foundation of China (12101144,  12001125).
}


\author[]{Hongfu Yang}
\ead{yanghf783@nenu.edu.cn}
\address{School of Mathematics and Statistics, Guangxi Normal University, Guangxi, Guilin, 541004, PR China}

\begin{abstract}
The well-known stochastic SIS model characterized by highly nonlinear in
epidemiology has a unique positive solution taking values in a bounded domain with a series of dynamical behaviors. However, the
approximation methods to maintain the positivity and long-time behaviors for the stochastic SIS model, while very important, are also lacking.
 In this paper,  based on a logarithmic transformation, we propose a novel explicit numerical method for a stochastic SIS epidemic model whose coefficients violate the global monotonicity condition, which can preserve the positivity of the original stochastic SIS model. And we show the strong convergence of the numerical method and derive that the rate of convergence is of order one. Moreover, the extinction of the exact solution of stochastic SIS model is reproduced. Some numerical experiments are given to illustrate the theoretical results and testify the efficiency of our algorithm.
\end{abstract}


\end{frontmatter}


\section{Introduction}\label{s-w}

Infectious diseases are the public enemy of human population and seriously threatening the health and life. In many research methods of infectious disease, epidemic models are widely used for analyzing the spread and control of infectious diseases. A common type of mathematical models in epidemiology is known as the SIS (Susceptible-Infected-Susceptible) epidemic models, which are used to describe some infections that don't have permanent immunity and individuals become susceptible again after infection. For example,
Hethcote-Yorke \cite{Hethcote}  has indicated that rota-viruses, some sexually transmitted and bacterial infections have such characteristics.
 In the real world, epidemic spreads are always affected by demographic and environmental stochasticity. However, deterministic model can't capture these fluctuating features. Hence, the stochastic SIS epidemic model \cite{GGHMP2011}, which can provide description and prediction more accurately than that of their deterministic counterpart,
 have attracted considerable attention and is described by
\begin{align}\label{(4)}
\left\{\begin{aligned}
\mathrm{d}S(t)&= \big[\mu N-\beta S(t)I(t)+\gamma I(t)-\mu S(t)\big]\mathrm{d}t-\sigma S(t)I(t)\mathrm{d}B(t),\\
\mathrm{d}I(t)&= \big[\beta S(t)I(t)-(\mu+\gamma)I(t)\big]\mathrm{d}t
+\sigma S(t)I(t)\mathrm{d}B(t),
\end{aligned}\right.
\end{align}
with initial values $I(0)+S(0)=S_0+I_0=N$, where $S_{t}$ denotes the number of susceptibles, $I_{t}$ is the number of infecteds at time $t$
and $N$ is the total amount of entire population among whom the disease is spreading. Moreover, $\mu\geq0$ is the per capita death rate, and $\gamma\geq0$ is the cure rate of infected individuals, so $1/\gamma$ is the average infectious period. In addition, $\beta\geq0$ is the disease transmission coefficient, and $B(t)$ is a standard Brownian motion defined on a complete probability space. $\sigma$ is the variance of the number of potentially infectious contacts that a single infected individual makes with another individual over the small time interval $[t,t +\mathrm{d}t)$.

It is worth noting that it is easy to get $I(t)+S(t)=N$ by solving
$$
\mathrm{d}\big[S(t)+I(t)\big]=\big[\mu N-\mu\big( S(t)+I(t)\big)\big]\mathrm{d}t.
$$
Therefore, \eqref{(4)} can be written equivalently as follows
\begin{align}\label{2.2*}
\mathrm{d}I(t) = \big[ \eta I(t)-\beta I(t)^2\big]\mathrm{d}t
+\sigma  I(t)\big(N-I(t)\big)\mathrm{d}B(t),
\end{align}
with initial value $I(0)=I_0\in (0, N)$, where $\eta=\beta N-\mu-\gamma$,
and Gray et al. \cite{GGHMP2011} have studied a number of nice properties for the special stochastic SIS epidemic model \eqref{2.2*}. Particularly,
here we mention an important theorem from \cite{GGHMP2011}, which shows that the unique solution of \eqref{2.2*} not only exists but also remains within the bounded domain $(0,N)$.
\begin{theorem}
For any given initial value $I_0\in (0,N)$, the SDE \eqref{2.2*} has a unique global solution taking values in the domain $(0,N)$, i.e.,
$$
\mathbb{P}\big\{I(t)\in (0,N)~\forall t\geq 0\big\}=1.
$$
\end{theorem}
Nevertheless, the true solution of the nonlinear SDE \eqref{2.2*} is not available, which means that looking for
well-performed numerical methods is necessary to simulate and help to observe the behavior of the underlying SIS epidemic model. The most common numerical method for SDEs with global Lipschitz continuous coefficients is classical Euler-Maruyama (EM) method. But Hutzenthaler-Jentzen-Kloeden \cite[Theorem 2.1]{16} have showed that, for SDEs whose coefficients are growing superlinearly, EM approximations would inevitably involve the $p$th moments diverging to infinity. In recent years, several modified EM schemes, including implicit schemes and explicit schemes, have been proposed for approximating the exact solutions of nonlinear SDEs, see, e.g., \cite{Szpruch2011,Li2021,Mao2013Szpruch,Tretyakov2013,Hutzenthaler12,
Hutzenthaler15,Sabanis13,Sabanis16,Wang13,LiuW13,Mao20152} and the references therein. Even so, these methods mentioned above with requiring a global monotonicity condition cannot be directly used for \eqref{2.2*}.

On the other hand, constructing appropriate numerical approximations, which can perserve the properties of the exact solution as well as reproduce the long-time behaviors for SDEs, has been one of the research focuses in numerical analysis field. For example,
Li-Yang \cite{li_yang2021} have applied a positivity-preserving truncated EM (TEM) method to the stochastic logistic model, and further realized
long-time dynamical properties, including the stochastic permanence, extinctive and stability in distribution. Recently,  Yi-Hu-Zhao \cite{Yi2021} have presented a positivity-preserving logarithmic EM scheme for general SDEs whose solutions are positive.  Mao-Wei-Wiriyakraikul \cite{Mao2021} have established a new positive preserving TEM method for
stochastic Lotka-Volterra competition model. Moreover,
  there are few results of the positivity-preserving numerical methods for stochastic SIS model \eqref{2.2*}. Specifically,  Chen-Gan-Wang \cite{CGW2021} have constructed the Lamperti smoothing truncation scheme for \eqref{2.2*}, which can preserve the domain of the original SDEs and have a convergence rate of order 1. And then Yang-Huang \cite{Yang2021} have developed a positivity preserving logarithmic EM method, which can be regarded as an improvement on the \cite{CGW2021}.
 Nonetheless, to the best of our knowledge, \cite{Yi2021,Mao2021,li_yang2021,CGW2021,Yang2021} has not yet provided the results of the long-time dynamical properties of the numerical approximation for stochastic SIS model \eqref{2.2*}. Precisely speaking, there is still no result about the numerical methods which not only preserve the positivity but also the long-time dynamical properties for stochastic SIS model \eqref{2.2*}. Thus, we in the paper are devoted to presenting an explicit positivity-preserving numerical scheme and show the extinction of the numerical method.

Inspired by the works in \cite{Alfonsi05, Neuenkirch14,CGW2021}, we make use of the Lamperti transformation to construct a type of logarithmic-type transformed numerical method for \eqref{2.2*}. To be more specific, applying the It\^{o} formula to $y(t)$ in the Lamperti-type transformation
\begin{align}\la{ts}
y(t)=\log(I(t))-\log(N-I(t))=\log\frac{I(t)}{N-I(t)},
 \end{align}
we can arrive at
 \begin{align}\la{yhf*}
\mathrm{d}y(t)=F(y(t))\mathrm{d}t+\sigma N \mathrm{d}B(t),~~~~t\geq 0,
\end{align}
with $y(0)=\log(I_0)-\log(N-I_0)$, where $F(x)=\eta-(\mu+\gamma)e^{x}+\frac{\sigma^2 N^2}{2}
-\frac{\sigma^2 N^2}{1+e^x}$.
Note that \eqref{ts} also means
\begin{align}\la{2.4*}
 I(t)=N-\frac{N}{1+e^{y(t)}}=\frac{Ne^{y(t)}}{1+e^{y(t)}}.
 \end{align}
More recently, Li-Mao-Yang \cite{Li2021} proposed a new explicit method and obtained the  asymptotic stability of the method under more relaxed conditions. However, the transformed SDE  \eqref{yhf*} is still not satisfied with some conditions in \cite{Li2021}. To deal with transformed SDE \eqref{yhf*} with exponentially growing coefficients, in this paper,  we establish a logarithmic TEM scheme, which can be transformed back to an positivity-preserving approximation of the original SDE \eqref{2.2*} according to \eqref{2.4*}. This allows us to preserve the same domain $(0,N)$ of the original SDE \eqref{2.2*} and obtain the first order rate of convergence. Moreover, one of our important contributions in our paper is the approximation of long-time dynamical properties for \eqref{2.2*}, namely, the extinction of the numerical method is given.




The rest of the paper is  structured as follows. Section \ref{n-p} introduces some preliminary results on certain
properties for the stochastic SIS model.  Section \ref{section3} constructs an explicit scheme preserving the positivity and the domain of the  exact solution,  and yields  the convergence rate of first order.  Section \ref{section5} focuses on the analysis the extinction of the numerical scheme, i.e. the explicit scheme preserving  the extinction of the exact solution of stochastic SIS model. Section \ref{s5} provides a couple of numerical examples and simulations to demonstrate the efficiency of the numerical scheme.

\section{Mathematical notations and lemmas}\label{n-p}
Throughout this paper,  we  use the following notations.  Let $( \Omega,~\cal{F}$,~$\PP )$ be a complete  probability space, and $\mathbb{E}$ denotes the expectation corresponding to $\PP$, and   $|\cdot|$ denote the Euclidean norm in $\RR$.
   For any $a, b\in \mathbb{R}$,
 $a\vee b:=\max\{a,b\}$, and $a\wedge b:=\min\{a,b\}$.
If $\mathbb{D}$ is a set, its indicator function is denoted by $I_{\mathbb{D}}$, namely $I_{\mathbb{D}}(x)=1$ if $x\in \mathbb{D}$ and $0$ otherwise. Suppose ${\{ {\cal{F}}_{t}\}} _{t \geq 0}$ is a filtration defined on this probability space satisfying the usual conditions (i.e., it is right continuous and $\mathcal{F}_0$ contains all $\mathbb{P}$-null sets) such that $B(t)$
is $ {\cal{F}}_{t}$ adapted. First,  we state some useful lemmas which can be found in \cite{CGW2021}.
\begin{lemma}[{\cite[Theorem 3.2]{CGW2021}}]\la{CGW3.2}
 For any $p\geq 0$ we have
\begin{align*}
\bigg(\sup_{t\in [0,T]}\mathbb{E}\big[I(t)^{-p}\big]\bigg)\vee
\bigg(\sup_{t\in [0,T]}\mathbb{E}\big[\big(N-I(t)\big)^{-p}\big]\bigg)\leq K_p,
\end{align*}
where
\begin{align}\la{CGW3.8}
K_p=\Big(I_0^{-p}\vee \big(N-I_0\big)^{-p}\Big)\exp{\Big(p T\big(|\beta N-\mu-\gamma|+2\beta N\big)+\frac{p(p+1)}{2}\sigma^2N^2T\Big)}.
\end{align}
\end{lemma}

\begin{lemma}[{\cite[Proposition 3.3]{CGW2021}}]\la{le2.3}
 For any $p\in \mathbb{R}$, the transformed SDE \eqref{yhf*} has the following exponential integrability property
\begin{align*}
 \sup_{t\in [0,T]}\mathbb{E}\big[e^{py(t)}\big]  \leq N^{|p|}K_{|p|},
\end{align*}
where $K_p$ is given by \eqref{CGW3.8}.
\end{lemma}

Next we propose EM method to approximate the exact solution of the SDE \eqref{yhf*}.
Given a step size $\Delta>0$ and let $t_k=k\Delta$ for any integer $k\geq 0$.
 For convenience, we will further use $C_u$ and $C$ to stand for two generic positive real constants dependent on $I_0, \beta, N, \mu, \gamma, \sigma, p$ but independent of the step size $\t$ and $k$, which we will use in this section and later sections.
Please note that the values of $C_u$ and $C$ may change between occurrences,  where
$C_u$ is dependent on $u$.  For SDE  \eqref{yhf*}  define  the EM  scheme by
\begin{align}\label{yhf_EM}
\left\{
\begin{array}{ll}
Y_0=y(0),&\\
Y^E_{k+1} = Y^E_k+F(Y^E_k)\t+\sigma N \t B_k,
\end{array}
\right.
\end{align}
for any integer $k\geq 0$, where  $\t B_k=B(t_{k+1})-B(t_{k})$. Transforming back, i.e.
\begin{align}\la{3.2*}
 I^{E,\t}_k=N-\frac{N}{1+e^{Y^E_k}}=\frac{Ne^{Y^E_k}}{1+e^{Y^E_k}},
 \end{align}
gives a strictly positive approximation of the original  SDE \eqref{2.2*}.

\begin{lemma}[{\cite[Lemma 3.2]{li_yang2021}}]\la{leB}
For any $L>0$ and integer $m\geq 0$, we have
\begin{align}\la{log:eq:5.16}
\E\Big[\mathrm{e}^{\frac{|\Delta B_k|^2}{4\Delta}}|\Delta B_k|^{3m}\Big]
 = \sqrt{\frac{2}{\pi}} \Gamma\Big(\frac{3m+1}{2}\Big)\big(4\Delta\big)^{\frac{3m}{2}},~~~~~
 \E\big[\Theta_k^m\big|\mathcal{F}_{t_k}\big]
 \leq C \Delta^{\frac{3m}{2}},
\end{align}
where $\Gamma(\cdot)$ is the Gamma function and
$
\Theta_k:=\mathrm{e}^{L |\Delta B_k|}\big( \Delta^{\frac{3}{2}}+  |\Delta B_k|^3\big).
$
\end{lemma}

\begin{lemma}[{\cite[Theorem 2.1]{Yang2021}}]\la{SDE_log_th2}
For any $p>0$ there exists a constant $C>0$  such that
\begin{align}\la{*}
\mathbb{E}\bigg[\sup_{k=0,\ldots, \lfloor T/\Delta\rfloor} |Y^E_{k}-y(t_{k})|^{p}\bigg]
\leq& C_T\Delta^{p}
\end{align}
for any $\Delta\in(0, 1]$ and $T>0$,  where $\lfloor T/\Delta\rfloor$  represents the integer part of $T/\Delta$.
\end{lemma}

\begin{rem}
Although Lemma \ref{SDE_log_th2} shows that the logarithmic EM method can strongly converge to the exact solution of stochastic SIS model \eqref{2.2*}, which can be found in \cite[Theorem 2.2]{Yang2021}, it has not yet provided the results about the approximation of the long-time dynamical properties for stochastic SIS model \eqref{2.2*}.
The main reason is that the diffusion term and nonlinear term of SDE \eqref{yhf_EM} will lead to extraordinary values. In other words, a numerical method which can efficiently prevent producing the extraordinary values would be able to preserve the properties of the exact solution on the basis of Taylor expansion. Thus, in the next section, we define the explicit scheme by using an appropriate truncation map.
\end{rem}
\section{Convergence rate of the logarithmic TEM method}\la{section3}
In this section, we aim to construct a logarithmic  TEM method, which can preserve the positivity of
the original stochastic SIS model \eqref{2.2*}. Moreover,  this allows us to preserve the same domain $(0,N)$ of the original stochastic SIS model \eqref{2.2*} and obtain the first order rate of convergence.
 For any given stepsize $\Delta\in(0, 1)$, define a  TEM   scheme by
\begin{align}\label{**}
\left\{\begin{aligned}
\bar{Y}_{k+1}&= Y_k+F(Y_k)\t+\sigma N \t B_k,\\
Y_{k+1}&= \bar{Y}_{k+1}\wedge \log(K\t^{-\frac{1}{2}}),
\end{aligned}\right.
\end{align}
for any integer $k\geq0$, where  $K>1+e^{Y_0}$.
 Transforming back, i.e.
\begin{align}\la{4.2*}
\bar{I}^{\t}_k=N-\frac{N}{1+e^{\bar{Y}_k}}=\frac{Ne^{\bar{Y}_k}}{1+e^{\bar{Y}_k}},~~~~I_k^{\t}=N-\frac{N}{1+e^{Y_k}}=\frac{N e^{Y_k}}{1+e^{Y_k}},
\end{align}
and $\bar{I}^{\t}_k$ gives a strictly positive approximation of the original  SDS \eqref{2.2*}. Lemmas \ref{log_le3.1} and \ref{xinle1} will play their important
roles in the proof of  the strong convergence of the numerical solutions.

\begin{lemma}\la{log_le3.1}
For any $p>0$, the TEM scheme defined by \eqref{**} has the property that
\begin{align}
\sup_{\Delta\in (0, 1)}\sup_{0\leq k\leq \lfloor T/\Delta\rfloor}\mathbb{E}\big[\mathrm{e}^{pY_{k}}\big]
\leq \sup_{\Delta\in (0, 1)}\sup_{0\leq k\leq \lfloor T/\Delta\rfloor}\mathbb{E}\big[\mathrm{e}^{p\bar{Y}_{k}}\big]
\leq  C_T
\end{align}
for any $T>0$.
\end{lemma}
\begin{proof}
We know that
\begin{align*}
\bar{Y}_{k+1}\leq Y_{k}+ \big(\eta+\frac{\sigma^2 N^2}{2}\big)
\Delta +\sigma N \Delta B_k\leq \bar{Y}_{k}+ \big(\eta+\frac{\sigma^2 N^2}{2}\big)
\Delta +\sigma N \Delta B_k.
\end{align*}
Then
 $$
\bar{Y}_{k}\leq  Y_0+\big(\eta+\frac{\sigma^2 N^2}{2}\big) k\Delta
  +\sum_{i=0}^{k-1} \sigma N\Delta B_{i}
$$
for any integer $k\geq 1$. Thus, for any $p>0$  we have
\begin{align*}
\mathbb{E}\big[\mathrm{e}^{p\bar{Y}_{k}}\big]
\leq&\mathbb{E}\Big[\exp\Big(p Y_0 +p\big(\eta+\frac{\sigma^2 N^2}{2}\big)T+p \sum_{i=0}^{k-1} \sigma N\Delta B_{i}\Big)\Big]\nn\\
 \leq & \exp\Big(p Y_0 +p\big(\eta+\frac{\sigma^2 N^2}{2}\big)T
+\frac{p^2\sigma^2N^2}{2}T\Big).
\end{align*}
The proof is complete.
\end{proof}


\begin{lemma}\la{xinle1}
For any $p> 0$, the TEM scheme defined by \eqref{**} has the property that
\begin{align}
\sup_{\Delta\in (0, 1)}\sup_{0\leq k\leq \lfloor T/\Delta\rfloor}\mathbb{E}\big[\mathrm{e}^{-pY_{k}}\big]
\leq  C_T
\end{align}
for any $\Delta\in(0, 1]$ and $T>0$.
\end{lemma}
\begin{proof}
Using the well-known  Taylor formula we get
\begin{align*}
\mathrm{e}^{-\bar{Y}_{k+1}}\leq& \mathrm{e}^{-Y_{k}}-\mathrm{e}^{-Y_{k}}(\bar{Y}_{k+1}-Y_{k})+\frac{1}{2}\mathrm{e}^{-Y_{k}}(\bar{Y}_{k+1}-Y_{k})^2
+\frac{1}{6}\mathrm{e}^{-Y_{k}}
\mathrm{e}^{|\bar{Y}_{k+1}-Y_{k}|}|\bar{Y}_{k+1}-Y_{k}|^3,
\end{align*}
where $|\bar{Y}_{k+1}-Y_k|
\leq   C\t + (\mu+\gamma) e^{Y_k}\t +|\sigma N|  |\t B_k|
\leq  C \t^{\frac{1}{2}}+|\sigma N| |\t B_k|$, and
\begin{align}\la{eqyhf4.5}
|\bar{Y}_{k+1}-Y_k|^3=&|F(Y_k)\t+\sigma N \t B_k|^3\leq 4\big(|F(Y_k)|^3\t^3+|\sigma N|^3 |\t B_k|^3\big)\nn\\
\leq&  4\big(4|\eta+0.5 \sigma^2 N^2|^3\t^3+4(\mu+\gamma)^3e^{3Y_k}\t^3+|\sigma N|^3 |\t B_k|^3\big)\nn\\
\leq&  4\big(4|\eta+0.5 \sigma^2N^2|^3\t^3+4K^3(\mu+\gamma)^3 \t^{3/2}+|\sigma N|^3 |\t B_k|^3\big).
\end{align}
 Then using the above inequality,   we have
\begin{align*}
\mathrm{e}^{-\bar{Y}_{k+1}}\leq& \mathrm{e}^{-Y_{k}}\Big[1-(\bar{Y}_{k+1}-Y_{k})
+\frac{1}{2}(\bar{Y}_{k+1}-Y_{k})^2
 +C\exp{\big(C \t^{\frac{1}{2}}+|\sigma N| |\t B_k| \big)}\big(\t^{\frac{3}{2}}+|\t B_k|^3\big)\Big]\nn\\
 \leq& \mathrm{e}^{-Y_{k}}\Big[1-(\bar{Y}_{k+1}-Y_{k})
+\frac{1}{2}(\bar{Y}_{k+1}-Y_{k})^2
 +C\exp{\big(|\sigma N| |\t B_k|\big)}\big(\t^{\frac{3}{2}}+|\t B_k|^3\big)\Big]\nn\\
 =&\mathrm{e}^{-Y_{k}}\Big[1-F(Y_k)\t-\sigma N \t B_k+\frac{1}{2}\big(F(Y_k)\t+\sigma N \t B_k\big)^2+U_k\Big]\nn\\
\leq&\mathrm{e}^{-Y_{k}}\Big[1-F(Y_k)\t-\sigma N \t B_k+ \big|F(Y_k)\big|^2\t^2+|\sigma N|^2 \big|\t B_k\big|^2+U_k\Big]\nn\\
\leq&\mathrm{e}^{-Y_{k}}\Big[1+\big((\mu+\gamma)e^{Y_k}-\eta\big)\t-\sigma N \t B_k+ C\t^2+K^2(\mu+\gamma)\t+|\sigma N|^2 \big|\t B_k\big|^2+U_k\Big],
\end{align*}
where
$$
U_k=C\exp\big( |\sigma N| |\Delta B_k|\big)\big( \Delta^{\frac{3}{2}}+  |\Delta B_k|^3\big),
$$
by Lemma \ref{leB} one observes that
\begin{align}\la{log:eq:5.16}
\E\big[U_k^m\big|\mathcal{F}_{t_k}\big]
 \leq C \Delta^{\frac{3m}{2}}.
\end{align}
Then
\begin{align}\la{logeq:5.45}
(1+\mathrm{e}^{-\bar{Y}_{k+1}})^{p}\leq (1+\mathrm{e}^{-Y_{k}})^{p}(1+\zeta_k)^{p},
\end{align}
where
\begin{align*}
\zeta_k=& \frac{\mathrm{e}^{-Y_{k}}}{1+\mathrm{e}^{-Y_{k}}}\Big[ (\mu+\gamma)e^{Y_k}\t+\big(K^2(\mu+\gamma)-\eta\big)\t+C \Delta^{2}
+ |\sigma N|^2  (\Delta B_k)^2 -\sigma N \Delta B_k+U_k\Big],
\end{align*}
and we can see that $\zeta_k>-1$.
For the given constant $p>2$, choose an integer $m$ such that $2m<p\leq2(m+1)$. It follows from \cite[Lemma 3.3]{li2018jde} and \eqref{logeq:5.45} that
\begin{align}\la{logeq:5.46}
&\E \Big[\big(1+\mathrm{e}^{-\bar{Z}_{k+1}}\big)^{p}  \big|\mathcal{F}_{t_k}\Big]\nn\\
\leq&\big(1+\mathrm{e}^{-Z_{k}}\big)^{p} \bigg\{1+ p \E\big[\zeta_k\big|\mathcal{F}_{t_k}\big] + \frac{p(p-1)}{2} \E\big[\zeta_k^2\big|\mathcal{F}_{t_k}\big] + \mathbb{E}\big[P_m(\zeta_k)\zeta_k^3 \big|\mathcal{F}_{t_k}\big]\bigg\},
\end{align}
where $P_m(x)$ represents a $m$th-order polynomial of $x$  with coefficients depending only on $p$, and $m$ is an
integer.  Noticing that the increment $\t B_k$ is independent of $\mathcal{F}_{t_k}$,  we
derive that
\begin{align}\la{properties}
\mathbb{E}\big[(\t B_k)^{2j-1}\big|\mathcal{F}_{t_k}\big]=0,~~~~
\mathbb{E}\big[(\t B_k)^{2j}\big|\mathcal{F}_{t_k}\big]=\frac{(2j)!}{2^jj!}\t^j,~~  j= 1,2,\ldots,
\end{align}
and using   \eqref{log:eq:5.16}, we compute
\begin{align}\la{Y_4}
\mathbb{E}\big[\zeta_k \big|\mathcal{F}_{t_k}\big]
=&\frac{\mathrm{e}^{-Y_{k}}}{1+\mathrm{e}^{-Y_{k}}}\Big( (\mu+\gamma)e^{Y_k}\t+\big(K^2(\mu+\gamma)+|\sigma N|^2-\eta\big)\t+C \Delta^{2} +\mathbb{E}\big[U_k\big|\mathcal{F}_{t_k}\big]\Big)\nn\\
\leq& \frac{\mathrm{e}^{-Y_{k}}}{1+\mathrm{e}^{-Y_{k}}}\Big( (\mu+\gamma)e^{Y_k}\t+C\t \Big)\leq C\t,
\end{align}
and
\begin{align} \la{Y_5}
\mathbb{E}\big[\zeta_k^2 \big|\mathcal{F}_{t_k}\big]
=&\frac{\mathrm{e}^{-2Y_{k}}}{\big(1+\mathrm{e}^{-Y_{k}}\big)^2}
\mathbb{E}\Big[\Big( (\mu+\gamma)e^{Y_k}\t+\big(K^2(\mu+\gamma)-\eta\big)\t+C \Delta^{2}\nn\\
&~~~~~~~~~~~~~~~~~~~~+ |\sigma N|^2  (\Delta B_k)^2 -\sigma N \Delta B_k+U_k\Big)^2\Big|\mathcal{F}_{t_k}\Big]\nn\\
\leq& \frac{\mathrm{e}^{-2Y_{k}}}{\big(1+\mathrm{e}^{-Y_{k}}\big)^2}
\Big( 3(\mu+\gamma)^2e^{2Y_k}\t^2+C\t^2 +3|\sigma N|^2 \t+\mathbb{E}\big[U_k^2\big|\mathcal{F}_{t_k}\big]\Big)\nn\\
\leq& \frac{\mathrm{e}^{-2Y_{k}}}{\big(1+\mathrm{e}^{-Y_{k}}\big)^2}
\Big( 3(\mu+\gamma)^2e^{2Y_k}\t^2+C\t\Big)\leq C\t.
\end{align}
To estimate $\mathbb{E}\big[P_m(\zeta_k)\zeta_k^3 \big|\mathcal{F}_{t_k}\big]$, we begin with $\mathbb{E}\big[\zeta_k^3 \big|\mathcal{F}_{t_k}\big]$. Using \eqref{log:eq:5.16} and \eqref{properties} we obtain
 \begin{align}
\mathbb{E}\big[\zeta_k^3 \big|\mathcal{F}_{t_k}\big]
=&\frac{\mathrm{e}^{-3Y_{k}}}{\big(1+\mathrm{e}^{-Y_{k}}\big)^3}
\mathbb{E}\Big[\Big( (\mu+\gamma)e^{Y_k}\t+\big(K^2(\mu+\gamma)-\eta\big)\t+C \Delta^{2}\nn\\
&~~~~~~~~~~~~~~~~~~~~+ |\sigma N|^2  (\Delta B_k)^2 -\sigma N \Delta B_k+U_k\Big)^3\Big|\mathcal{F}_{t_k}\Big]\nn\\
\leq& \frac{\mathrm{e}^{-3Y_{k}}}{\big(1+\mathrm{e}^{-Y_{k}}\big)^3}
\Big(  (\mu+\gamma)^3e^{3Y_k}\t^3+C\t^{\frac{3}{2}}  +\mathbb{E}\big[U_k^3\big|\mathcal{F}_{t_k}\big]\Big)
\leq  C\t^{\frac{3}{2}}.
\end{align}
 On the other hand, we can use the same method to derive that
\begin{align}
\mathbb{E}\big[\zeta_k^3 \big|\mathcal{F}_{t_k}\big]
=&\frac{\mathrm{e}^{-3Y_{k}}}{\big(1+\mathrm{e}^{-Y_{k}}\big)^3}
\mathbb{E}\Big[\Big( (\mu+\gamma)e^{Y_k}\t+\big(K^2(\mu+\gamma)-\eta\big)\t+C \Delta^{2}\nn\\
&~~~~~~~~~~~~~~~~~~~~+ |\sigma N|^2  (\Delta B_k)^2 -\sigma N \Delta B_k+U_k\Big)^3\Big|\mathcal{F}_{t_k}\Big]\nn\\
\geq& -\frac{\mathrm{e}^{-3Y_{k}}}{\big(1+\mathrm{e}^{-Y_{k}}\big)^3}
\Big(  (\mu+\gamma)^3e^{3Y_k}\t^3+C\t^{\frac{3}{2}}  +\mathbb{E}\big[U_k^3\big|\mathcal{F}_{t_k}\big]\Big)
\geq   -C\t^{\frac{3}{2}}.
\end{align}
 Thus, both of the above inequality imply $\mathbb{E}\big[a_0\zeta_k^3\big|\mathcal{F}_{t_k}\big]\leq o(\t)$ for any constant $a_0$, where $a_i$ represents the coefficient of $\zeta_k^i$ term in polynomial $P_m(\zeta_k)$. We can also show that
\begin{align*}
\mathbb{E}\big[|a_i\zeta_k^{3+i}|\big|\mathcal{F}_{t_k}\big]\leq o(\t)
\end{align*}
for any $i\geq 1$. These implies
\begin{align}\la{Y_8}
\mathbb{E}\big[P_m(\zeta_k)\zeta_k^3 \big|\mathcal{F}_{t_k}\big]\leq o(\t).
\end{align}
Combining \eqref{logeq:5.46}, \eqref{Y_4}, \eqref{Y_5} and \eqref{Y_8},   we obtain that
\begin{align}
 \E \Big[\big(1+\mathrm{e}^{-\bar{Y}_{k+1}}\big)^{p}  \big|\mathcal{F}_{t_k}\Big]
\leq \big(1+\mathrm{e}^{-Y_{k}}\big)^{p}\big(1+C\t\big)
\end{align}
for any integer $0\leq k\leq \lfloor T/\Delta\rfloor$.  Obviously,
 \begin{align}
 \E \Big[\big(1+\mathrm{e}^{-\bar{Y}_{k+1}}\big)^{p}  \Big]
\leq \big(1+C\t\big)\E \Big[\big(1+\mathrm{e}^{-Y_{k}}\big)^{p}\Big].
\end{align}
Define
$\Omega_k=\big\{\omega\big|\bar{Y}_{k}\leq \log(K\t^{-\frac{1}{2}})\big\}$ and using the Chebyshev's inequality, we can  see that
\begin{align}
 \E \Big[\big(1+\mathrm{e}^{-Y_{k}}\big)^{p}  \Big]=& \E \Big[\big(1+\mathrm{e}^{-Y_{k}}\big)^{p}I_{\Omega_k}  \Big]+\E \Big[\big(1+\mathrm{e}^{-Y_{k}}\big)^{p}I_{\Omega^c_k}  \Big]\nn\\
=& \E \Big[\big(1+\mathrm{e}^{-\bar{Y}_{k}}\big)^{p}I_{\Omega_k}  \Big]+\E \Big[\big(1+K^{-1}\t^{\frac{1}{2}}\big)^{p}I_{\Omega^c_k}  \Big]\nn\\
\leq & \E \Big[\big(1+\mathrm{e}^{-\bar{Y}_{k}}\big)^{p}  \Big]+2^p\mathbb{P}\big\{\bar{Y}_{k}> \log(K\t^{-\frac{1}{2}})\big\}\nn\\
\leq &\big(1+C\t\big)\E \Big[\big(1+\mathrm{e}^{-Y_{k-1}}\big)^{p}\Big]
+2^p\frac{\mathbb{E}e^{2\bar{Y}_{k}}}{K^2\t^{-1}}.
\end{align}
 It follows from the result  of  Lemma \ref{log_le3.1} that
\begin{align}
 \E \Big[\big(1+\mathrm{e}^{-Y_{k}}\big)^{p}  \Big]
\leq &\big(1+C\t\big)\E \Big[\big(1+\mathrm{e}^{-Y_{k-1}}\big)^{p}\Big]
+C\t\nn\\
\leq &\big(1+C\t\big)^k \big(1+\mathrm{e}^{-Y_{0}}\big)^{p}
+C\t\sum_{i=0}^{k-1}\big(1+C\t\big)^i\nn\\
\leq &e^{Ck\t} \big(1+\mathrm{e}^{-Y_{0}}\big)^{p}
+Ck\t e^{Ck\t} \leq e^{CT}\Big[\big(1+\mathrm{e}^{-Y_{0}}\big)^{p}+CT\Big].
\end{align}
Therefore  the desired result follows.
\end{proof}

Define
$
\tau_{R}=\inf\{k\Delta\geq 0: y(t_k)\geq \log R\},~ \rho_{\Delta}=\inf\big\{k\Delta\geq 0: \bar{Y}_{\Delta}(t_k)\geq \log(K\Delta^{-\frac{1}{2}} )\big\}.
$
By Lemma \ref{le2.3}   we have
\begin{align*}
R^{p}\mathbb{P}\{\tau_{R}\leq T\}\leq  \mathbb{E}\Big[e^{py(T\wedge \tau_{R})}I_{\{\tau_{R}\leq T\}}\Big] +\mathbb{E}\Big[e^{py(T\wedge \tau_{R})}I_{\{\tau_{R}> T\}}\Big]=\mathbb{E}\big[e^{py(T\wedge \tau_{R})}\big]\leq C_T.
\end{align*}
It is easy to see that
\begin{align} \la{logeq3.5}
\mathbb{P}\{\tau_{R}\leq T\}\leq  C_T R^{-p},~~~~~\forall~p>0,
\end{align}
  where $C$  is  a positive constant independent of $R$.   Moreover, by the virtue of Lemma \ref{log_le3.1}, we have
\begin{align*}
 \E\Big[ \exp{\big(p \log(K\Delta^{-\frac{1}{2}} )  \big)}I_{\{\rho_{\Delta}  \leq  T \}}\Big]
 \leq& \E\Big[ \exp{\big(p\bar{Y}_{\Delta}( T \wedge \rho_{\Delta})  \big)}\Big] =\mathbb{E}\Big[\exp{\big(p\bar{Y}_{  \lfloor\frac{T\wedge  \rho_{\Delta}}{\t}\rfloor}\big)}\Big]\nn\\
 \leq&\sup_{  0\leq k\leq \lfloor T/\t \rfloor}\mathbb{E}\big[\exp{\big(p\bar{Y}_{k}\big)}\big]
\leq C_T,
\end{align*}
implies that
\begin{align} \la{logeq3.6}
 \mathbb{P} {\big\{\rho_{\t} \leq   T \big\}}\leq    C_T K^{-p}\Delta^{ p/2},~~~~\forall~p>0,
\end{align}
  where $C$  is  a positive constant independent of $\Delta$.

\begin{theorem}\la{logTh3.6}
For any $p>0$, the TEM scheme defined by \eqref{**} has the property that
\begin{align*}
\sup_{0\leq k\leq\lfloor {T}/{\Delta}\rfloor}\mathbb{E}\Big [  |Y_{k}-y(t_{k})|^{p}\Big]
\leq  C_T\Delta^p
\end{align*}
for any $\Delta\in (0, 1)$ and $T>0$.
\end{theorem}
\begin{proof}Define $\bar{\theta}_{ \t}=\tau_{K\Delta^{-1/2 }} \wedge \rho_{\t}$, $  \Omega_1:= \{\omega: \bar{\theta}_{\t} > T\}$, $\bar{u}_k=Y_k-y(t_{k}),$ for  any $k\Delta\in[0, T]$,
where  $\tau_L$ and $\rho_{\t}$  are defined by \eqref{logeq3.5}, and \eqref{logeq3.6},  respectively.
For any $q>p$, using the Young inequality we obtain that
\begin{align}\la{logeq3.7}
\E|\bar{ u}_{k}|^p =& \E\lf(|\bar{ u}_{k}|^p I_{\Omega_1}\rt)
+ \E\lf(|\bar{ u}_{k}|^p I_{\Omega_1^c}\rt) \nonumber \\
 \leq &\E\lf(|\bar{ u}_{k}|^p I_{\Omega_1}\rt)
+ \frac{ p\t^{p}}{q}\E\lf(|\bar{ u}_{k}|^{q}  \rt)+ \frac{q-p}{q\t^{p^2/(q-p)}} \PP(\Omega_1^c).
\end{align}
It follows from the results of  Lemmas \ref{le2.3}, \ref{log_le3.1} and \ref{xinle1}  that
\begin{align}\la{logeq3.8}
 \frac{p\t^p}{q}\E\lf(|\bar{ u}_{k}|^{q}  \rt) \leq  C_T\t^p.
\end{align}
 It follows from  \eqref{logeq3.5}, and \eqref{logeq3.6} that
\begin{align}\la{logeq3.9}
 \frac{q-p}{q\t^{p^2/(q-p)}} \PP(\Omega_1^c)  \leq &  \frac{q-p}{q\t^{p^2/(q-p)}}  \Big( \PP {\{\tau_{K\Delta^{-1/2 }} \leq T\}} +\PP {\{\rho_{  \t} \leq T\}}  \Big)\nonumber\\
 \leq & \frac{2(q-p)}{q\t^{p^2/(q-p)}}   \frac{C_T}{ K^{\delta}\Delta^{-\frac{\delta}{2}}}
\leq C_T
 \Delta^{\frac{\delta}{2}  -\frac{p^2}{q-p}}   \leq  C_T\t^p,
\end{align}
where $ \delta\geq  2pq/ (q-p)$.
Inserting \eqref{logeq3.8}, \eqref{logeq3.9}  and \eqref{*} into \eqref{logeq3.7} yields
\begin{align}\la{cond-13}
\E|\bar{ u}_{k}|^p
 \leq\E\lf(|\bar{ u}_{k}|^p I_{\Omega_1}\rt)
+  C_T\Delta^p
\leq&
\sup_{0\leq k\leq\lfloor \frac{T\wedge \rho_{\Delta}}{\Delta}\rfloor}\mathbb{E}\Big[  |Y_{k}-y(t_{k})|^{p}\Big]+  C_T\Delta^p\nn\\
=&
\sup_{0\leq k\leq\lfloor \frac{T\wedge \rho_{\Delta}}{\Delta}\rfloor}\mathbb{E}\Big[  |Y^E_{k}-y(t_{k})|^{p}\Big]+  C_T\Delta^p
\leq  C_T\Delta^p.
\end{align}
The proof is complete.  \end{proof}

\begin{theorem}\la{ }
For any $p>0$ there exists a constant $C>0$   such that
\begin{align*}
\sup_{k=0,\ldots, \lfloor T/\Delta\rfloor}\mathbb{E}\Big[ \big|I(t_k)-I^{\t}_k\big|^{p}\Big]
\leq& C_T\Delta^{p}
\end{align*}
for any $\Delta\in(0, 1]$ and $T>0$.
\end{theorem}
\begin{proof}   By the Lamperti-transformation \eqref{2.4*} and \eqref{4.2*}, we have
\begin{align*}
 \sup_{k=0,\ldots, \lfloor T/\Delta\rfloor}\mathbb{E}\bigg[ \Big|I(t_k)-I^{\t}_k\Big|^{p}\bigg]
=&N\sup_{k=0,\ldots, \lfloor T/\Delta\rfloor}\mathbb{E}\bigg[ \Big|\frac{\mathrm{e}^{y(t_k)}}{1+\mathrm{e}^{y(t_k)}}
-\frac{\mathrm{e}^{Y_k}}{1+\mathrm{e}^{Y_k}}\Big|^{p}\bigg]\nn\\
\leq& N\sup_{k=0,\ldots, \lfloor T/\Delta\rfloor}\mathbb{E}\bigg[ \big|y(t_k)-Y_k\big|^{p}\bigg].
\end{align*}
Thus, by applying Theorem \ref{logTh3.6}, we infer that
\begin{align}\la{}
\mathbb{E}\bigg[ \sup_{k=0,\ldots, \lfloor T/\Delta\rfloor}\Big|I(t_k)-I^{\t}_k\Big|^{p}\bigg]
\leq  N\sup_{k=0,\ldots, \lfloor T/\Delta\rfloor} \mathbb{E}\bigg[ \big|y(t_k)-Y_k\big|^{p}\bigg]\leq C_T\Delta^{p}
\end{align}
for any $\Delta\in(0, 1]$. The proof is complete.
\end{proof}

\section{Extinction}\la{section5}
In the study of population systems, extinction is one  of the most important issues, the easily
implementable scheme preserving the underlying extinction is desired eagerly. In this
section we cite the extinction criterion of the exact solution for SIS epidemic models \eqref{2.2*} (see, e.g.,  \cite{GGHMP2011})
and go a further step to show the logarithmic TEM method keeping extinction well.
To begin this section, we cite an extinction of stochastic SIS models,  please see details in \cite{GGHMP2011}.

\begin{theorem}[{\cite[Theorem 4.1]{GGHMP2011}}]\la{T:5.1}
 If
\begin{align}\la{mao4.1}
R_0^S:=R^D_0-\frac{\sigma^2N^2}{2(\mu+\gamma)}=\frac{\beta N}{\mu+\gamma}-\frac{\sigma^2N^2}{2(\mu+\gamma)}<1~~~~and ~~~~\sigma^2\leq \frac{\beta}{N},
\end{align}
then for any given initial value $I_0\in(0,N)$, the solution of the SDE \eqref{2.2*} obeys
\begin{align}
\limsup_{t\rightarrow \infty}\frac{1}{t}\log(I(t))\leq \beta N-\mu-\gamma-0.5\sigma^2N^2<0~~~~a.s.;
\end{align}
namely, $I(t)$ tends to zero exponentially a.s. In other words, the disease dies out with probability one.
\end{theorem}

\begin{theorem}\la{yhfTh4.2}
Under the condition  of Theorem \ref{T:5.1},
there is a constant $\t \in (0,  \Delta^{*})$ such that the scheme
 \eqref{**} has the property that
\begin{align}\la{mao4.2*}
\limsup_{k\t\rightarrow \infty}\frac{1}{k\t}\log(I^{\t}_k)\leq  \beta N-\mu-\gamma-0.5\sigma^2N^2+h(\t)<0~~~~a.s.;
\end{align}
namely, $I^{\t}_k$ tends to zero exponentially a.s. In other words, the disease dies out with probability one, where
\begin{align}\la{h_4.4}
h(\t):= 6|\eta+0.5 \sigma^2N^2|^3\t^{2}+6K^3(\mu+\gamma)^3 +4|\sigma N|^3\t^{\frac{1}{2}},
\end{align}
and $\t^*$ is the positive solution of the equation (in $\t$)
$$
h(\t)=0.5\sigma^2N^2+\mu+\gamma-\beta N.
$$
\end{theorem}
\begin{proof} Clear, we can see that $Y_k\leq \log(K\t^{-\frac{1}{2}})$ and  $Y_{k}\leq\bar{Y}_{k}.$  Direct calculation leads to
$$
I_{k}^{\t}-\bar{I}_k^{\t}=\frac{N}{1+e^{\bar{Y}_k}}-\frac{N}{1+e^{Y_k}}
=\frac{N(e^{Y_k}-e^{\bar{Y}_k})}{(1+e^{\bar{Y}_k})(1+e^{Y_k})}\leq 0,
$$
implies that $I_k\leq \bar{I}_k$. Here and below we always denote
$$f(x):=N-\frac{N}{1+e^{x}}=\frac{N e^x}{1+e^x},~~ \forall~x\in \mathbb{R},$$
we derive that
\begin{align}
\left\{\begin{aligned}
 &f'(x)=\frac{N e^x}{1+e^x}\frac{1}{1+e^x}=\frac{f(x)}{1+e^x},~~~f''(x)=\frac{N e^x(1-e^x)}{(1+e^x)^3}=\frac{f(x)(1-e^x)}{(1+e^x)^2},\\
 &f'''(x)=\frac{N e^x}{(1+e^x)^4}\big[(1-e^x)^2-2e^x\big]
=\frac{f(x)}{(1+e^x)^3}\big[(1-e^x)^2-2e^x\big].
\end{aligned}\right.
\end{align}
To proceed, we define
$$V(x):=\log f(x)=\log \frac{N e^x}{1+e^x},~~ \forall~x\in \mathbb{R}.$$
 It is easy to see that
\begin{align}
\left\{\begin{aligned}
&V'(x)=\frac{f'(x)}{f(x)}=\frac{1}{1+e^x},\\
 &V''(x)=\big[\frac{f'(x)}{f(x)}\big]'
=\frac{f''(x)}{f(x)}-\Big(\frac{f'(x)}{f(x)}\Big)^2
= \frac{1-e^x}{(1+e^x)^2}-\frac{1}{(1+e^x)^2}=\frac{-e^x}{(1+e^x)^2},\\
&V'''(x)=\frac{f'''(x)}{f(x)}-3\frac{f'(x)f''(x)}{f^2(x)}
+2\Big[\frac{f'(x)}{f(x)}\Big]^3=\frac{(1-e^x)^2}{(1+e^x)^3}
-\frac{1}{(1+e^x)^2},
\end{aligned}\right.
\end{align}
and for any $x\in \mathbb{R}$,
\begin{align}\la{eqyhf5.6}
|V'''(x)|\leq \frac{1}{(1+e^x)}+\frac{1}{(1+e^x)^2}\leq 2.
\end{align}
Then, by the virtue of the Taylor formula with Lagrange remainder term, we have
\begin{align}
V(\bar{Y}_{k+1})=&V(Y_k)+V'(Y_k)(\bar{Y}_{k+1}-Y_k)
+\frac{V''(Y_k)}{2}(\bar{Y}_{k+1}-Y_k)^2\nn\\
&+\frac{ V'''(Y_k+\theta(\bar{Y}_{k+1}-Y_k))}{3!}(\bar{Y}_{k+1}-Y_k)^3\nn\\
=&V(Y_k)+\frac{\bar{Y}_{k+1}-Y_k}{1+e^{Y_k}}
+ \frac{-e^{Y_k}(\bar{Y}_{k+1}-Y_k)^2}{2(1+e^{Y_k})^2}
+\frac{ V'''(Y_k+\theta(\bar{Y}_{k+1}-Y_k))}{3!}(\bar{Y}_{k+1}-Y_k)^3,
\end{align}
where $\theta\in (0,1)$. By \eqref{eqyhf5.6}, we can see that
\begin{align}\la{eqyhf5.8}
V(\bar{Y}_{k+1})
\leq &V(Y_k)+(1+e^{Y_k})^{-1}(\bar{Y}_{k+1}-Y_k)
+ \frac{-e^{Y_k}}{2(1+e^{Y_k})^2}(\bar{Y}_{k+1}-Y_k)^2
+\frac{1}{3}|\bar{Y}_{k+1}-Y_k|^3.
\end{align}
It follows from Lemma \ref{leB} that
$
\mathbb{E}|\t B_k|^3=(2\t)^{\frac{3}{2}}/\sqrt{\pi},
$
and by \eqref{eqyhf4.5}, one observes
\begin{align*}
\frac{1}{3}|\bar{Y}_{k+1}-Y_k|^3
\leq&   h(\t)\t +M_k^{(1)},
\end{align*}
where
$$
M_k^{(1)}=2|\sigma N|^3\Big(|\t B_k|^3-\frac{(2\t)^{\frac{3}{2}}}{\sqrt{\pi}}\Big).
$$
Inserting the above inequality into \eqref{eqyhf5.8} we have
\begin{align}\la{eqyhf5.9}
V(\bar{Y}_{k+1})
\leq &V(Y_k)+(1+e^{Y_k})^{-1}(\bar{Y}_{k+1}-Y_k)
+ \frac{-e^{Y_k}}{2(1+e^{Y_k})^2}(\bar{Y}_{k+1}-Y_k)^2
 +h(\t)\t +M_k^{(1)}.
\end{align}
The second and third terms at the right end of the estimation inequality, we have
\begin{align}\la{eqyhf5.10}
 (1+e^{Y_k})^{-1}(\bar{Y}_{k+1}-Y_k)
 =&(1+e^{Y_k})^{-1}\Big[\big(\eta-(\mu+\gamma)e^{Y_k}+\frac{\sigma^2N^2}{2}
-\frac{\sigma^2N^2}{1+e^{Y_k}}\big)\t+\sigma N\t B_k\Big]\nn\\
=&(1+e^{Y_k})^{-1}\big(\eta-(\mu+\gamma)e^{Y_k}\big)\t+(1+e^{Y_k})^{-1}\sigma N\t B_k\nn\\
&+(1+e^{Y_k})^{-1}\Big(\frac{\sigma^2N^2}{2}
-\frac{\sigma^2N^2}{1+e^{Y_k}}\Big)\t,
\end{align}
and
\begin{align*}
(\bar{Y}_{k+1}-Y_k)^2=&\Big[\big(\eta-(\mu+\gamma)e^{Y_k}+\frac{\sigma^2N^2}{2}
-\frac{\sigma^2N^2}{1+e^{Y_k}}\big)\t+\sigma N\t B_k\Big]^2\nn\\
=&\bigg(\eta-(\mu+\gamma)e^{Y_k}+\frac{\sigma^2N^2}{2}
-\frac{\sigma^2N^2}{1+e^{Y_k}}\bigg)^2\t^2+\sigma^2 N^2\big(|\t B_k|^2-\t\big)+\sigma^2 N^2\t\nn\\
&+2\sigma  N \t\bigg(\eta-(\mu+\gamma)e^{Y_k}+\frac{\sigma^2N^2}{2}
-\frac{\sigma^2N^2}{1+e^{Y_k}}\bigg)\t B_k\nn\\
=&\sigma^2 N^2\t+\bigg(\eta-(\mu+\gamma)e^{Y_k}+\frac{\sigma^2N^2}{2}
-\frac{\sigma^2N^2}{1+e^{Y_k}}\bigg)^2\t^2+M_k^{(2)},
\end{align*}
where
$$
M_k^{(2)}:= \sigma^2 N^2\big(|\t B_k|^2-\t\big) +2\sigma N \t\bigg(\eta-(\mu+\gamma)e^{Y_k}+\frac{\sigma^2N^2}{2}
-\frac{\sigma^2N^2}{1+e^{Y_k}}\bigg)\t B_k.
$$
 Then
\begin{align}\la{eqyhf5.11}
\frac{-e^{Y_k}}{2(1+e^{Y_k})^2}(\bar{Y}_{k+1}-Y_k)^2=&
\frac{-e^{Y_k}}{2(1+e^{Y_k})^2}\Big[\big(\eta-(\mu+\gamma)e^{Y_k}
+\frac{\sigma^2N^2}{2}
-\frac{\sigma^2N^2}{1+e^{Y_k}}\big)\t+\sigma N\t B_k\Big]^2\nn\\
=&-\frac{\sigma^2 N^2e^{Y_k}}{2(1+e^{Y_k})^2}\t-\frac{e^{Y_k}}{2(1+e^{Y_k})^2}M_k^{(2)}\nn\\
&-\frac{e^{Y_k}}{2(1+e^{Y_k})^2}\bigg(\eta-(\mu+\gamma)e^{Y_k}+\frac{\sigma^2N^2}{2}
-\frac{\sigma^2N^2}{1+e^{Y_k}}\bigg)^2\t^2
\end{align}
Substituting \eqref{eqyhf5.10} and \eqref{eqyhf5.11} into \eqref{eqyhf5.9} yields
\begin{align}\la{eqyhf5.12}
V(\bar{Y}_{k+1})
\leq &V(Y_k)+(1+e^{Y_k})^{-1}\big(\eta-(\mu+\gamma)e^{Y_k}\big)\t+(1+e^{Y_k})^{-1}\sigma N\t B_k\nn\\
&+(1+e^{Y_k})^{-1}\Big(\frac{\sigma^2N^2}{2}
-\frac{\sigma^2N^2}{1+e^{Y_k}}\Big)\t -\frac{\sigma^2 N^2e^{Y_k}}{2(1+e^{Y_k})^2}\t-\frac{e^{Y_k}}{2(1+e^{Y_k})^2}M_k^{(2)}\nn\\
&-\frac{e^{Y_k}}{2(1+e^{Y_k})^2}\bigg(\eta-(\mu+\gamma)e^{Y_k}+\frac{\sigma^2N^2}{2}
-\frac{\sigma^2N^2}{1+e^{Y_k}}\bigg)^2\t^2 +h(\t)\t +M_k^{(1)}\nn\\
\leq &V(Y_k)+(1+e^{Y_k})^{-1}\big(\eta-(\mu+\gamma)e^{Y_k}\big)\t +(1+e^{Y_k})^{-1}\Big(\frac{\sigma^2N^2}{2}
-\frac{\sigma^2N^2}{1+e^{Y_k}}\Big)\t\nn\\
& -\frac{\sigma^2 N^2e^{Y_k}}{2(1+e^{Y_k})^2}\t  +h(\t)\t +M_k,
\end{align}
where
$$
M_k=(1+e^{Y_k})^{-1}\sigma N\t B_k-\frac{e^{Y_k}}{2(1+e^{Y_k})^2}M_k^{(2)}+M_k^{(1)}.
$$
We can see that
\begin{align*}
(1+e^{Y_k})^{-1}\Big(\frac{\sigma^2N^2}{2}
-\frac{\sigma^2N^2}{1+e^{Y_k}}\Big)-\frac{ \sigma^2N^2e^{Y_k}}{2(1+e^{Y_k})^2}
= \frac{\sigma^2N^2}{2(1+e^{Y_k})^2}\Big(1+e^{Y_k}-2-e^{Y_k}\Big)
=&-\frac{\sigma^2N^2}{2(1+e^{Y_k})^2},
\end{align*}
and
\begin{align*}
 (1+e^{Y_k})^{-1} \big(\eta-(\mu+\gamma)e^{Y_k}\big)
 =&  (1+e^{Y_k})^{-1} \big(\beta N-\mu-\gamma-(\mu+\gamma)e^{Y_k}\big)\nn\\
 =&  (1+e^{Y_k})^{-1} \big(\beta N-(\mu+\gamma)(1+e^{Y_k})\big)
 =   \frac{\beta N}{1+e^{Y_k}} -\mu-\gamma.
\end{align*}
Making use of the above inequality and \eqref{eqyhf5.12} yields
\begin{align}\la{eqyhf5.13}
V(\bar{Y}_{k+1})
\leq &V(Y_k)+\Big(\frac{\beta N}{1+e^{Y_k}} -\mu-\gamma-\frac{\sigma^2 N^2}{2(1+e^{Y_k})^2}\Big)\t +h(\t)\t +M_k.
\end{align}
Due to \eqref{4.2*} and definition of $V$, we have
$$V(\bar{Y}_{k+1})=\log(\bar{I}^{\t}_{k+1}),~V(Y_{k})=\log(I^{\t}_{k}),
~N-I^{\t}_k=\frac{N}{1+e^{Y_k}},$$
and hence $V(Y_{k+1})\leq V(\bar{Y}_{k+1}),$
\begin{align*}
  \frac{\beta N}{1+e^{Y_k}} -\mu-\gamma= \beta N- \beta I^{\t}_k-\mu-\gamma,~ -\frac{\sigma^2N^2}{2(1+e^{Y_k})^2}=-0.5\sigma^2(N-I^{\t}_k)^2.
\end{align*}
Then
\begin{align}\la{mao4.3}
 \log(I^{\t}_{k+1})
\leq &\log(I_k^{\t})+ \Big[\beta N-\mu-\gamma- \beta I_k^{\t}-0.5\sigma^2(N-I^{\t}_k)^2\Big]\t +h(\t)\t +M_k\nn\\
= &\log(I_k^{\t})+ g(I^{\t}_k)\t
 +h(\t)\t+M_k,
\end{align}
for any integer $k\geq 0$, where $g:\mathbb{R}\rightarrow \mathbb{R}$ is defined by
\begin{align}\la{mao4.4}
 g(x)=\beta N-\mu-\gamma- \beta x-0.5\sigma^2(N-x)^2.
\end{align}
However, under condition \eqref{mao4.1}, we have
$$
g(I_k^{\t})\leq \beta N-\mu-\gamma-0.5\sigma^2N^2
$$
for $I_k^{\t}\in (0,N)$. It then follows from \eqref{mao4.3} that
\begin{align*}
 \log(I^{\t}_{k+1})
\leq  &\log(I_k^{\t})+  \big(\beta N-\mu-\gamma-0.5\sigma^2N^2\big)\t
 +h(\t)\t+M_k.
\end{align*}
Repeating this procedure arrives at
\begin{align*}
 \log(I^{\t}_{k})
\leq  &\log(I_0)+  \big(\beta N-\mu-\gamma-0.5\sigma^2N^2\big)k\t
 +h(\t)k\t+\sum_{i=0}^{k-1}M_i.
\end{align*}
This implies
\begin{align}\la{mao4.6}
 \limsup_{k\t\rightarrow +\infty}\frac{1}{k\t}\log(I^{\t}_{k})
\leq   \beta N-\mu-\gamma-0.5\sigma^2N^2
 +h(\t) +\limsup_{k\t\rightarrow +\infty}\frac{1}{k\t}\sum_{i=0}^{k-1}M_i~~\mathrm{a.s.}
\end{align}
However, by the large number theorem for martingales (see, e.g., \cite{Mao08}), we have
$$
\limsup_{k\t\rightarrow +\infty}\frac{1}{k\t}\sum_{i=0}^{k-1}M_i=0~~\mathrm{a.s.}
$$
We therefore obtain the desired assertion \eqref{mao4.2*} from \eqref{mao4.6}.
\end{proof}

In Theorem \ref{T:5.1} we require the noise intensity $\sigma^2 \leq \beta/N$. The following theorem
covers the case when $\sigma^2 > \beta/N$.

\begin{theorem}[{\cite[Theorem 4.3]{GGHMP2011}}]\la{Th5.3}
If
\begin{align}\la{eqmao4.18}
\sigma^2>\frac{\beta}{N}\vee \frac{\beta^2}{2(\mu+\gamma)},
\end{align}
then for any given initial value $I_0\in (0,N)$, the solution of the SDE SIS model \eqref{2.2*} obeys
\begin{align}
\limsup_{t\rightarrow\infty}\frac{1}{t}\log(I(t))\leq -\mu-\gamma+\frac{\beta^2}{2 \sigma^2}<0~~~a.s.;
\end{align}
namely, $I(t)$ tends to zero exponentially a.s. In other words, the disease dies out with probability one.
\end{theorem}

Note that condition \eqref{eqmao4.18} implies that $R^S_0
\leq 1$.
\begin{theorem}
Under the condition  of Theorem \ref{Th5.3},
there is a constant $\t \in (0,  \Delta^{**})$ such that the scheme
 \eqref{**} has the property that
\begin{align}\la{mao4.2*}
\limsup_{k\t\rightarrow \infty}\frac{1}{k\t}\log(I^{\t}_k)\leq   -\mu-\gamma+ \frac{\beta^2}{2\sigma^2} +h(\t)<0~~~~a.s.;
\end{align}
namely, $I^{\t}_k$ tends to zero exponentially a.s. In other words, the disease dies out with probability one, where
 $\t^{**}$ is the positive solution of the equation (in $\t$)
$$
h(\t)=\mu+\gamma-\frac{\beta^2}{2\sigma^2},
$$
and $h(\t)$ is given by \eqref{h_4.4}.
\end{theorem}
\begin{proof}
We use the same way as in the proof of  Theorem \ref{yhfTh4.2}.   By condition \eqref{eqmao4.18}, it is easy to see that $g(x)\leq g(\bar{x})$ for any $x\in (0, N)$, where
$$\bar{x}:=\frac{\sigma^2N-\beta}{\sigma^2}\in (0, N).$$ Compute
\begin{align*}
 g(\bar{x})=\beta N-\mu-\gamma-0.5\sigma^2 N^2+\frac{(\sigma^2N-\beta)^2}{2\sigma^2}=-\mu-\gamma+ \frac{\beta^2}{2\sigma^2}<0.
\end{align*}
This implies, in the same way as in the proof of Theorem \ref{yhfTh4.2}, that
\begin{align*}
 \limsup_{k\t\rightarrow +\infty}\frac{1}{k\t}\log(I^{\t}_{k})
\leq   g(\bar{x})
 +h(\t)~~\mathrm{a.s.},
\end{align*}
as required. The proof is hence complete.
\end{proof}


%
%
%
%

\section{\large\bf Examples and simulations}\la{s5}
In order to illustrate the efficiency of the  logarithmic TEM  scheme  we introduce some examples and  simulations.
We shall assume that the unit of time is one day and the population sizes are measured in units of 1 million, unless otherwise stated.

\begin{expl}\la{exp4.1}{\rm
Consider the  stochastic SIS epidemic model (see, e.g., \cite[Example 4.2]{GGHMP2011})
\begin{align}\la{nu4.1}
\mathrm{d}I(t) = \big[5 I(t)-0.5 I(t)^2\big]\mathrm{d}t
+0.035  I(t)\big(100-I(t)\big)\mathrm{d}B(t).
\end{align}
Noting that
$$
R_0^S=\frac{\beta N}{\mu+\gamma}-\frac{\sigma^2N^2}{2(\mu+\gamma)}=1.111-0.136<1~~~~and ~~~~\sigma^2=0.001225\leq \frac{\beta}{N}=0.005,
$$
we can therefore conclude, by Theorem \ref{mao4.1}, that for any initial value $I_0\in
(0, 100)$, $I(t)$ will tend to zero exponentially with probability one.

To compare to the simulations of the classical EM method (see, \cite[Example 4.2]{GGHMP2011}), then  all parameters are  same as the classical EM method. The simulations in Figure \ref{path1} show the sample paths of the solution for $t\in [0,50]$ by the logarithmic TEM scheme.    Both figures show clearly that the logarithmic TEM scheme \eqref{4.2*}  reproduces the dynamic properties of the underlying SDE \eqref{nu4.1}.
\begin{figure}[!htb]
\begin{center}
\includegraphics[angle=0, height=6cm, width=17cm]{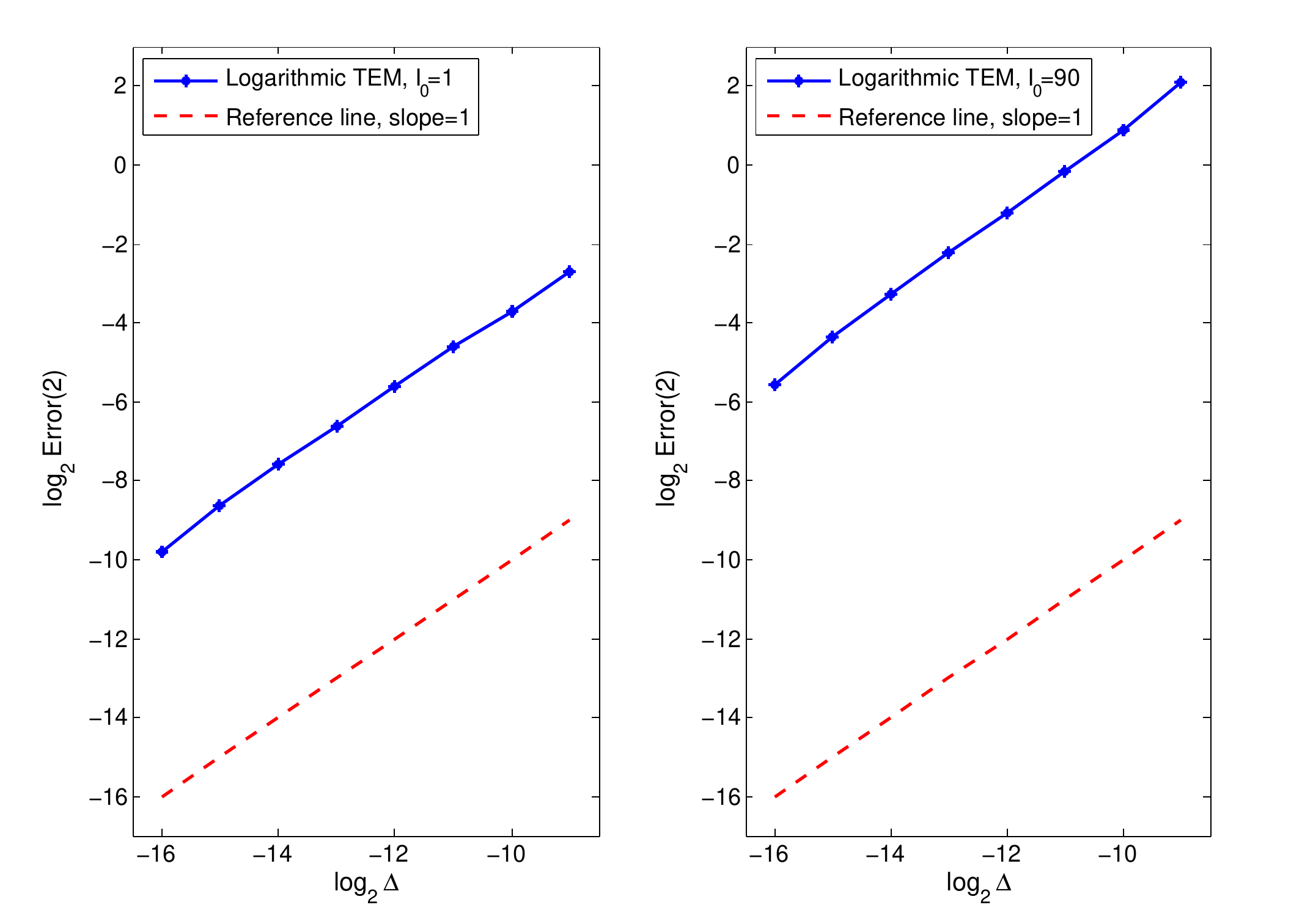}  
\caption{ The approximation error   of the exact solution and the numerical solution by  the logarithmic EM scheme \eqref{3.2*} as the function of step size $\t\in \{2^{-9},2^{-10},\dots,2^{-16}\}$, with initial values (left) $I_0=1$ and   (right) $I_0=90$.}\label{error1}
\end{center}                              
\vspace{-1em}
\end{figure}
\begin{figure}[!htb]
\begin{center}
\includegraphics[angle=0, height=6cm, width=17cm]{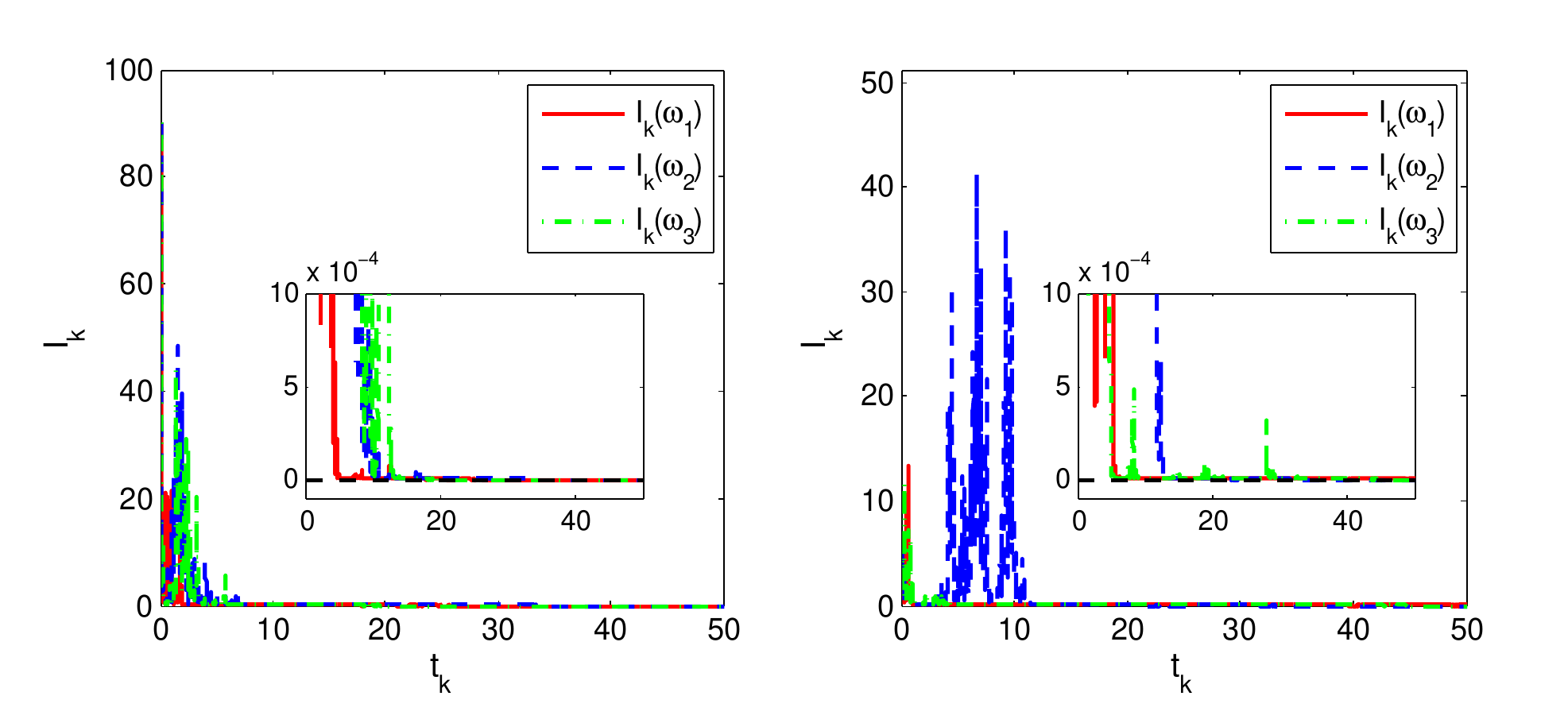}  
\caption{ The approximation error   of the exact solution and the numerical solution by  the logarithmic EM scheme \eqref{3.2*} as the function of step size $\t\in \{2^{-9},2^{-10},\dots,2^{-16}\}$, with initial values (left) $I_0=90$ and   (right) $I_0=1$.}\label{path1}
\end{center}                              
\vspace{-1em}
\end{figure}

 Due to no closed-form of the solution,
using scheme  \eqref{4.2*}, we regard  the better approximation with $\t=2^{-19}$  as  the exact solution $I(t)$ and compare it with the numerical solution $I_k$ with $\Delta=2^{-9}, 2^{-10}, \dots, 2^{-16}$. To compute the approximation error, we run $M$ independent trajectories where $I^{(j)}(t_k)$ and $I^{(j)}_k$ represent the $j$th trajectories of  the exact solution $I(t)$ and  the numerical solution $I_k$ respectively. Thus
$$
 \mathrm{Error}(p):=\bigg(\mathbb{E}\Big[\sup_{k=0,\ldots, \lfloor T/\Delta\rfloor}|I(t_k)-I_k|^p\Big]\bigg)^{\frac{1}{p}} \approx \bigg( \frac{1}{M}\sum_{j=1}^{M}\Big[
\sup_{k=0,\ldots, \lfloor T/\Delta\rfloor}|I^{(j)}(t_k)-I^{(j)}_k|^p\Big]\bigg)^{\frac{1}{p}}.
$$
To reduce the time of simulations without losing any necessary illustration, we only simulate the paths for $t\in[0,2]$, and
 plotting  the  $\log_{2}$-$\log_{2}$  approximation error $\big(\mathbb{E}\big[\sup_{k=0,\ldots, \lfloor 2/\Delta\rfloor}|I(t_k)-I_k|^5\big]\big)^{\frac{1}{5}}$   with $M=1000$.  The   blue solid line depicts  $\log_{2}$-$\log_{2}$ error   while the red dashed  is  a reference line of slope $1$ in Figure \ref{error1}.

}
\end{expl}

\begin{expl}{\rm
  We keep the system parameters the same as in Example \ref{exp4.1} but
let $\sigma = 0.08$, so the stochastic SIS model \eqref{2.2*} becomes (see, e.g., \cite[Example 4.4]{GGHMP2011})
\begin{align}
\mathrm{d}I(t) = \big[5 I(t)-0.5 I(t)^2\big]\mathrm{d}t
+0.08 I(t)\big(100-I(t)\big)\mathrm{d}B(t).
\end{align}
It is easy to verify that the system parameters obey condition \eqref{eqmao4.18}. We can therefore
conclude, by Theorem \ref{Th5.3}, that for any initial value $I_0\in (0, 100)$, $I(t)$ will tend to zero exponentially with probability one.

The computer
simulations shown in Figure \ref{error2} clearly support these results.
\begin{figure}[!htb]
\begin{center}
\includegraphics[angle=0, height=6cm, width=17cm]{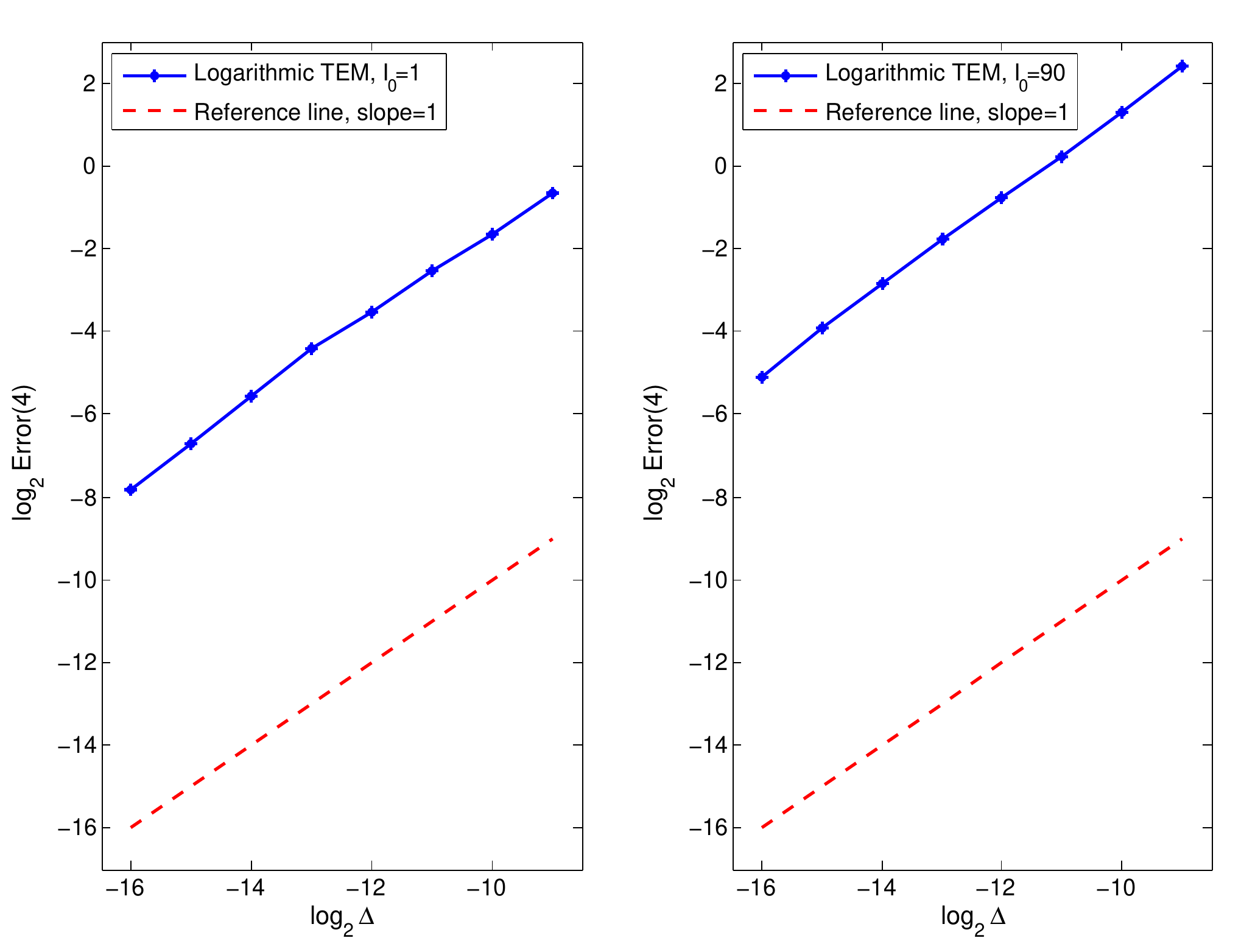}  
\caption{ The approximation error   of the exact solution and the numerical solution by  the logarithmic EM scheme \eqref{3.2*} as the function of step size $\t\in \{2^{-9},2^{-10},\dots,2^{-16}\}$, with initial values (left) $I_0=1$ and   (right) $I_0=90$.}\label{error2}
\end{center}                              
\vspace{-1em}
\end{figure}
\begin{figure}[!htb]
\begin{center}
\includegraphics[angle=0, height=6cm, width=17cm]{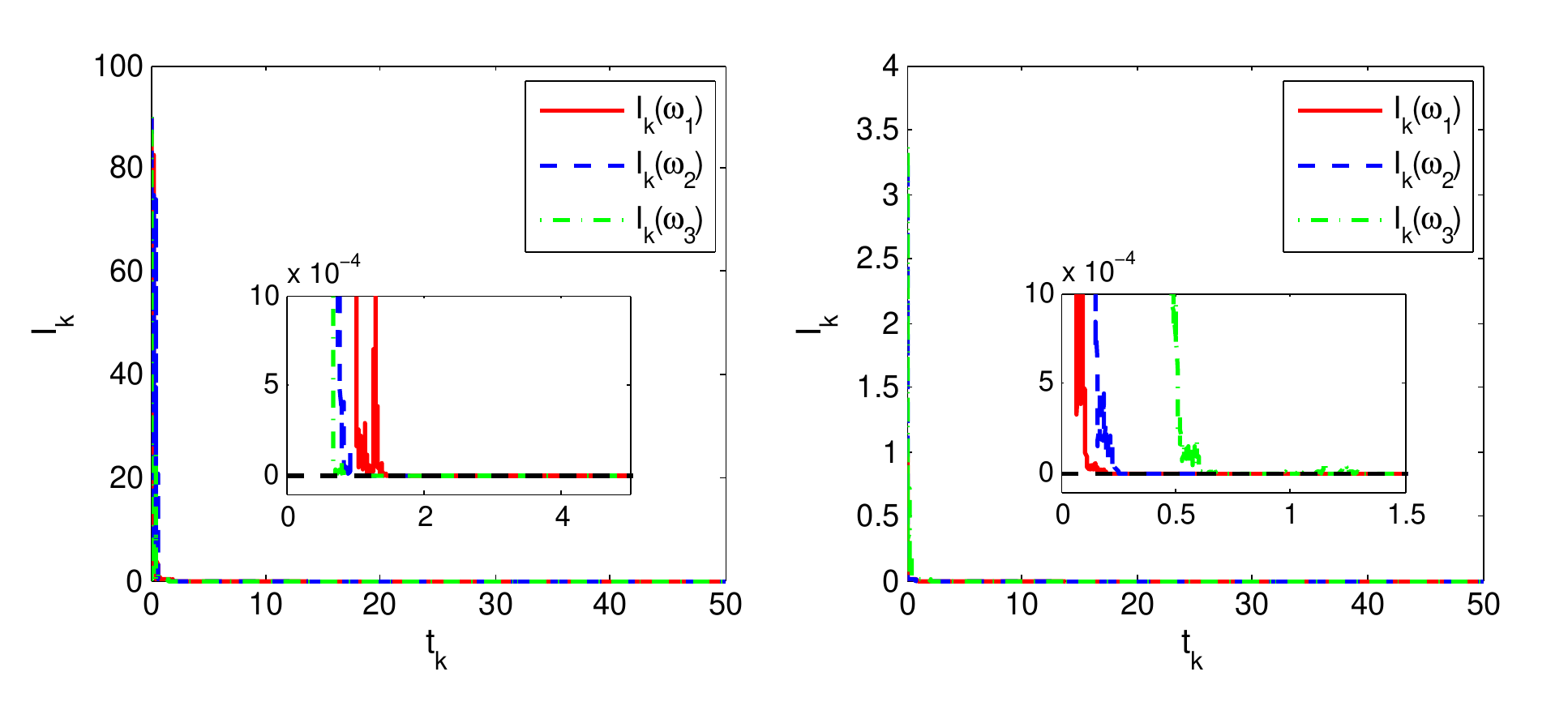}  
\caption{ The approximation error   of the exact solution and the numerical solution by  the logarithmic EM scheme \eqref{3.2*} as the function of step size $\t\in \{2^{-9},2^{-10},\dots,2^{-16}\}$, with initial values (left) $I_0=90$ and   (right) $I_0=1$.}\label{path2}
\end{center}                              
\vspace{-1em}
\end{figure}
}
\end{expl}

\begin{expl}{\rm
Consider the  stochastic SIS epidemic model (see, e.g., \cite[Example 5.3]{GGHMP2011})
\begin{align}\la{nu4.2}
\mathrm{d}I(t) = \big[5 I(t)-0.5 I(t)^2\big]\mathrm{d}t
+0.01  I(t)\big(100-I(t)\big)\mathrm{d}B(t).
\end{align}
The simulation  descriptions of this example are similar to Example \ref{exp4.1}, so we omit it here.
\begin{figure}[!htb]
\begin{center}
\includegraphics[angle=0, height=6cm, width=17cm]{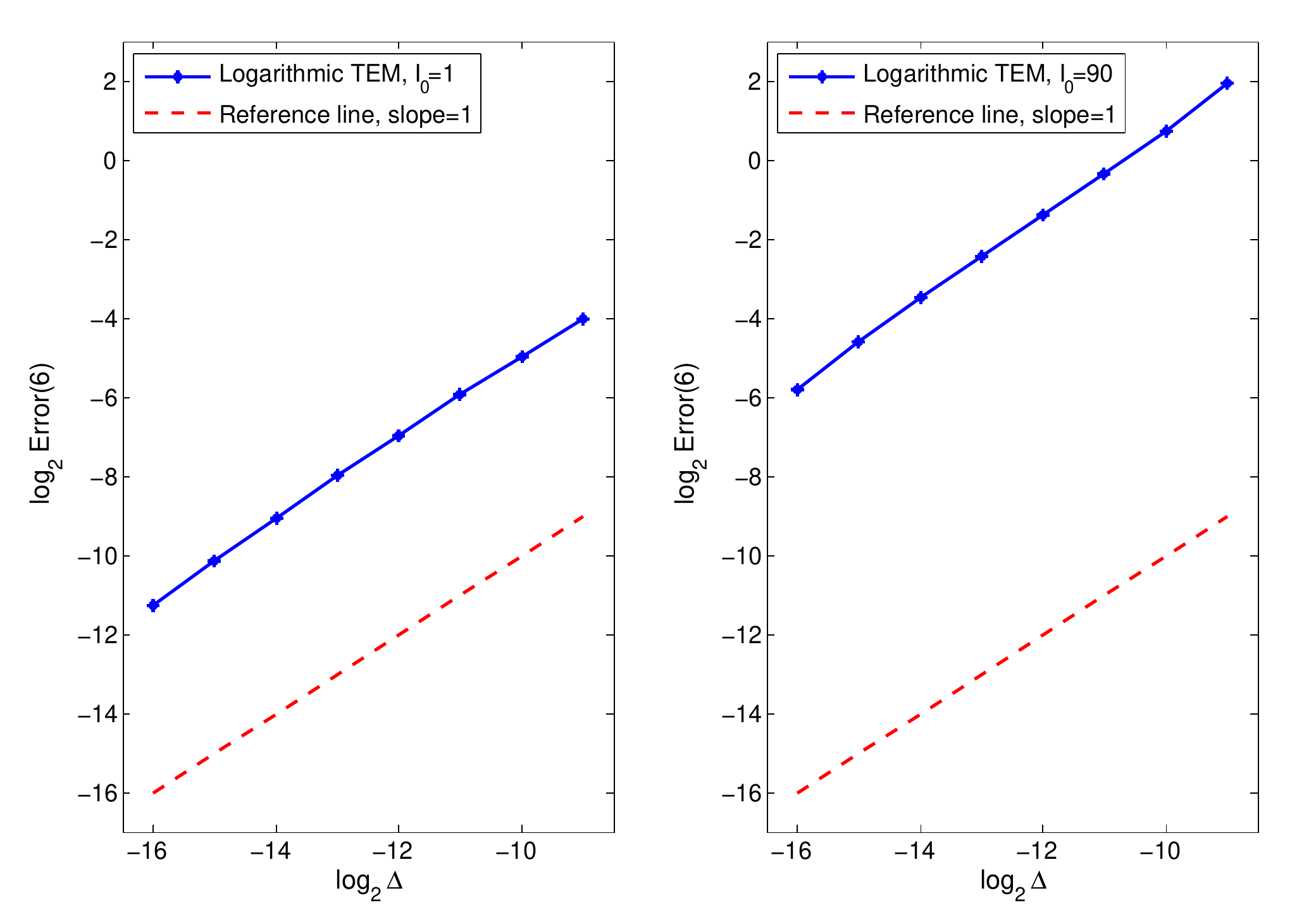}  
\caption{ The approximation error   of the exact solution and the numerical solution by  the logarithmic EM scheme \eqref{3.2*} as the function of step size $\t\in \{2^{-9},2^{-10},\dots,2^{-16}\}$, with initial values (left) $I_0=1$ and   (right) $I_0=90$.}\label{error3}
\end{center}                              
\vspace{-1em}
\end{figure}
\begin{figure}[!htb]
\begin{center}
\includegraphics[angle=0, height=6cm, width=17cm]{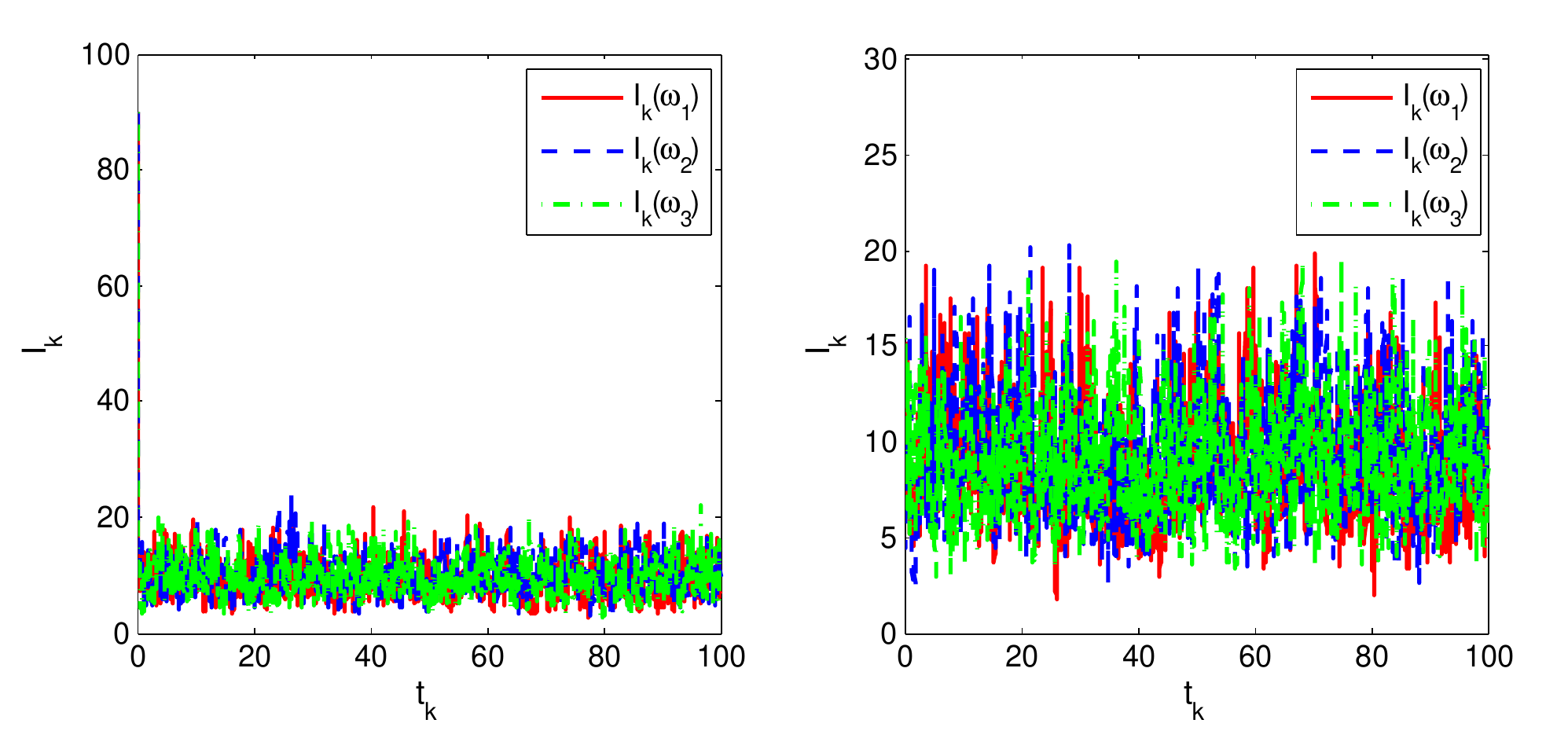}  
\caption{ The approximation error of the exact solution and the numerical solution by  the logarithmic EM scheme \eqref{3.2*} as the function of step size $\t\in \{2^{-9},2^{-10},\dots,2^{-16}\}$, with initial values (left) $I_0=90$ and   (right) $I_0=1$.}\label{path3}
\end{center}                              
\vspace{-1em}
\end{figure}
}
\end{expl}

\begin{rem}
This paper aims to show strong convergence rate and extinction of logarithmic TEM method for nonlinear SDEs \eqref{2.2*} under nonglobally monotone coefficients. However, we
have not obtained the persistence and stationary distribution of logarithmic TEM method,   which could be a new direction of our further work.  Although we have not proved the persistence of logarithmic TEM method, it is easy to see from Figure \ref{path3} that without extra restrictions the numerical solutions of the logarithmic TEM method  stay in step of persistence with the exact solutions.
\end{rem}

 \bibliography{mybibfile}

\end{document}